\documentclass[pdftex]{amsart}
\usepackage[utf8x]{inputenc}
\usepackage[T1]{fontenc}
\usepackage{ae,aecompl}

\usepackage{amssymb}
\usepackage{amsmath}

\usepackage{tikz}
\usetikzlibrary{matrix,shapes,arrows}

\usepackage{enumerate}
\usepackage{verbatim}
\usepackage{hyperref}

\usepackage{graphicx}

\usepackage{ifpdf}
\ifpdf
\fi

\newtheorem{theorem}{Theorem}[section]

\newtheorem{proposition}[theorem]{Proposition}
\newtheorem{corollary}[theorem]{Corollary}

\newtheorem{problem}[theorem]{Problem}

\newtheorem{conjecture}[theorem]{Conjecture}

\newtheorem{remark}[theorem]{Remark}

\numberwithin{equation}{section}

\newcommand{\sg}[1][n]{{\mathfrak{S}_{#1}}}
\newcommand{\suchthat}{\mid}
\newcommand{\NN}{\mathbb{N}}

\newcommand{\QQ}{\mathbb{Q}}
\newcommand{\CC}{\mathbb{C}}

\newcommand{\kk}{\mathbb{K}}

\newcommand{\harm}{\mathcal{D}}
\newcommand{\HW}{\mathcal{D}W}
\newcommand{\DW}[1][\ell]{\mathcal{D}W^{(#1)}}

\newcommand{\qHW}[2][q]{\mathcal{D}#2_{#1}}

\newcommand{\steen}{\operatorname{\mathcal{S}}}
\newcommand{\qsteen}[1][q]{\mathcal{S}_{#1}}
\newcommand{\tp}[2]{\langle #1 \rangle_{#2}}
\newcommand{\xy}[2]{x^{#1}y^{#2}}

\newcommand{\e}{\varepsilon}

\newcommand{\p}{\partial}

\newcommand{\X}{X}
\newcommand{\PX}[1][\X]{\partial_{#1}}
\newcommand{\XPX}[1][\X]{\X\!\cdot\PX}

\newcommand{\mx}[1][x]{{\mathbf{#1}}}
\renewcommand{\mp}{{\boldsymbol{\partial}}}
\newcommand{\md}{{\mathbf{d}}}
\newcommand{\inflate}{\phi}
\newcommand{\Inflate}{\Phi}
\newcommand{\inflatep}{\overline\inflate}
\newcommand{\Inflatep}{\overline\Inflate}
\newcommand{\hooktwoheadrightarrow}{\hookrightarrow\hspace{-1.9ex}\rightarrow}

\usepackage{ifthen}
\newboolean{draft}
\setboolean{draft}{false}
\ifdraft
\newcommand{\TODO}[2][To do: ]{\textcolor{red}{\textbf{#1#2}}}
\newcommand{\INFO}[2][Info: ]{\textcolor{red}{\textbf{#1#2}}}
\else
\newcommand{\TODO}[2][]{}
\newcommand{\INFO}[2][]{}
\fi

\definecolor{darkgreen}{rgb}{0.1,0.4,0.1}
\definecolor{darkred}{rgb}{0.4,0.1,0.1}
\definecolor{darkblue}{rgb}{0.1,0.1,0.4}
\begin{document}

\title[Deformed diagonal harmonic polynomials]{Deformed diagonal harmonic polynomials\\ for complex reflection groups}

\author{Fran\c cois Bergeron}
\address{D\'epartement de Math\'ematiques, UQAM,  
 C.P. 8888, Succ. Centre-Ville, 
 Montr\'eal,  H3C 3P8, Canada.}
\email{bergeron.francois@uqam.ca}

\author{Nicolas Borie}
\address{Univ. Paris-Sud, Laboratoire de Math\'ematiques d'Orsay,
 Orsay Cedex, F-91405; CNRS, France.}
\email{Nicolas.Borie@u-psud.fr}

\author{Nicolas M.~Thi\'ery}
\address{Univ. Paris-Sud, Laboratoire de Math\'ematiques d'Orsay,
 Orsay Cedex, F-91405; CNRS, France.}
\email{Nicolas.Thiery@u-psud.fr}


\maketitle

\begin{abstract}

  We introduce deformations of the space of (multi-diagonal) harmonic
  polynomials for any finite complex reflection group of the form
  $W=G(m,p,n)$, and give supporting evidence that this space seems to
  always be isomorphic, as a graded $W$-module, to the undeformed
  version.


\end{abstract}


\TODO{Acknowledge Maple + Sage}

\section{Introduction}

The aim of this work is to give support to an extension and a
generalization of the main conjecture of~\cite{Hivert_Thiery.SA.2002},
to the diagonal case as well as to the context of finite complex
reflection groups. This new extended conjecture is stated explicitly
in (\ref{la_conjecture}) below, after a few words concerning notations
and a description of the overall context.

Let $X$ denote a $\ell \times n$ matrix of variables
   $$X:=(\mx_1,\mx_2,\,\ldots,\, \mx_n),$$
with each of the columns $\mx_j=(x_{ij})_{1\leq i\leq \ell}$ containing $\ell$ variables. 
For any fixed $i$ (a row of $X$), we say that the variables $x_{i1},x_{i2},\ldots, x_{in}$ form a \emph{set of variables} (the $i^{\rm th}$ set),
and thus $X$ consists in $\ell$ sets of $n$ variables.
For $\md\in \NN^\ell$, we set
\begin{displaymath}
  |\md|:=d_1+d_2+\cdots +d_\ell
  \qquad \text{ and } \qquad
  \md!:=d_1!d_2!\cdots d_\ell!\, ,
\end{displaymath}
and write $\mx_j^\md$ for the column monomial of \emph{degree} $\mx_j^\md$:
$$\mx_j^\md:=\prod_{i=1}^\ell x_{ij}^{d_i}\ .$$
The ground field $\kk$ is assumed to be of characteristic zero and,
whenever needed, to contain roots of unity and/or a parameter $q$;
typically, $\kk=\CC$ or $\kk=\CC(q)$, although algebraic or
transcendental extensions of $\QQ$ are better suited for some of the
computer calculations. The parameter $q$ is called formal if it is
transcendental over $\QQ$.

Let $W$  be a complex reflection group of rank $n$. Elements of $W$ may be thought of as $n\times n$ matrices with complex entries.
The \emph{diagonal action} of $W$ on a polynomial $f(X)$ is defined, for $w \in W$, by:
\begin{equation}\label{defn_action}
  w\cdot f(X):=f(X\,w),
\end{equation}
where $X\,w$ stands for matrix multiplication. In other words, $W$
acts in a similar ``diagonal'' manner on each set of variables in $X$.
A polynomial is \emph{diagonally $W$-invariant} if
$$w\cdot f(X) =f(X),\qquad \hbox{for all}\quad w\in W.$$
We denote by $I_W^{(\ell)}$ the ideal generated by constant-term-free diagonally $W$-invariant polynomials.
 
For each of the variables $x_{ij} \in X$, there is an associated
partial derivation denoted here by $\p_{x_{ij}}$ or $\p_{ij}$ for
short. For a polynomial $f(X)$, let $f(\PX)$ stand for the
differential operator obtained by replacing variable in $\X$ by the
corresponding derivation in $\PX$.  The space $\DW$ of
\emph{diagonally $W$-harmonic polynomials} (or \emph{harmonic
  polynomials} for short) is then defined as the set of the
polynomials $g(X)$ that satisfy all of the linear partial differential
equations
\begin{equation}\label{defn_harmonic}
    f(\PX ) (g(X)) = 0,\qquad \hbox{for}\quad f(X)\in I_W^{(\ell)}\,.
\end{equation}
In the following, we first restrict ourselves to the complex
reflection groups $W=G(m,n)$, for $m,n\in \NN$, and then extend our
discussion to the subgroups $G(m,p,n)$. Recall that the
\emph{generalized symmetric group} $G(m,n)$ may be constructed as the
group of $n\times n$ matrices having exactly one non zero entry in
each row and each column, whose value is a $m$-th root of unity. Since
the cases $\ell=1$ and $W=\sg[n]$ have been extensively considered
previously (see~\cite{Bergeron.2010}), we write for short
$\HW=\DW[1]$, and $\harm_n=\harm\sg$.

The ring $\kk[X]^W$ of diagonally $W$-invariant polynomials for
$W=G(m,n)$ is generated by \emph{polarized} powersums, this is to say
the polynomials
   $$P_\md=\sum_{j=1}^n \mx_j^\md\,,$$
for $|\md|= m\,k$, with $1\leq k\leq n$.
Let us write $D_\md$ for the operator $P_\md(\PX )$:
$$D_\md=\sum_{j=1}^n \mp_j^\md\,,$$
where
$$\mp_j^\md:=\p_{1j}^{d_1}\p_{2j}^{d_2}\cdots \p_{\ell j}^{d_\ell}.$$
Then, the space $\DW$ is the intersection of the kernels of all the
operators $D_\md$, for $|\md|=m\,k$, with $1\leq k\leq n$. The space
$\DW$ is graded by (multi-)degree, and thus decomposes as a direct sum
    $$\DW=\bigoplus_{\md\in \NN^\ell} \DW_\md,$$
  of its homogeneous components of degree $\md$. Recall that $f(X)$ is homogeneous of degree $\md$, if and only if we have
     $$f(\mathbf{t}\,X)= \mathbf{t}^\md\,f(X),$$
 where $\mathbf{t}\,X$ stands for the multiplication of the matrix $X$ by the diagonal matrix
    $$\begin{pmatrix}  t_1 & \cdots &0\\
                                  \vdots &\ddots &\vdots\\
                                   0  & \cdots & t_\ell
          \end{pmatrix}.$$
The \emph{Hilbert series of the space $\DW$} is defined as
    $$\DW(\mathbf{t}):=\sum_{\md\in \NN^\ell} \dim(\DW_\md)\, \mathbf{t}^\md.$$
It is well known that, for $\ell=1$, the graded space $\HW$ is isomorphic, as a
$W$-module, to the graded regular representation of $W$. In particular,
its Hilbert series is given by the formula
$$\HW(t)=\prod_{k=1}^n \frac{t^{k\,m}-1}{t-1}\,.$$

Our specific story starts with a $q$-deformation of the polarized
powersums:
\begin{equation}\label{defn_q_P}
  P_{q,\md}:=\sum_{j=1}^n  \mx_j^\md\,(1 + q\,(x_{1j}\p_{1j} +\ldots + x_{nj}\p_{nj})).
\end{equation}
and the corresponding $q$-deformation of the operators
$D_\md=P_\md(\PX )$ above:
\begin{equation}\label{defn_q_D}
  D_{q,\md}:=   \sum_{j=1}^n (1 + q\,(x_{1j}\p_{1j} +\ldots + x_{nj}\p_{nj}))\, \mp_j^\md.
\end{equation}
An homogeneous polynomial $f(X)$ is said to be \emph{diagonally
  $W,q$-harmonic} (or $q$-harmonic for short) if
\begin{equation}
  D_{q,\md}\, f(X)=0,
\end{equation}
for all $\md\in \NN^\ell$ such that $|\md|$ is divisible by $m$.
The $W$-module of all $q$-harmonic polynomials is henceforth denoted
by $\DW_q$. Our aim here is to discuss and give support to the following:
\begin{conjecture}\label{la_conjecture}
  Take $W=G(m,n)$, $\ell\in\NN$, and $q$ formal. Then, $\DW_q$ is
  isomorphic (as a graded $W$-module) to $\DW$. In particular
  $\DW_q(\mathbf{t})=\DW(\mathbf{t})$.
\end{conjecture}

This conjecture is an extension of the main conjecture
of~\cite{Hivert_Thiery.SA.2002} (case $W=\sg[n]$, $\ell=1$), which
itself, at $q=1$, is equivalent to an original conjecture of
Wood~\cite{wood.psa.1998,wood.hpsa.2001} on the ``hit polynomials''
for the \emph{rational Steenrod algebra} $\steen:=\kk[P_{1,d}
\suchthat d\geq 1]$. Beside the extensive computer exploration and
results reported on in~\cite{Hivert_Thiery.SA.2002} for $W=\sg[n]$ and
$\ell=1$, our supporting evidence for this conjecture, includes the
following results:
\begin{enumerate}[(1)]
\item Applying a classical specialization argument at $q=0$ (see
  e.g.~\cite{Hivert_Thiery.SA.2002}), gives that $\dim \DW_q \leq \dim
  \DW$ (also homogeneous component by homogeneous
  component). Furthermore, equality holds if and only if
  Conjecture~\ref{la_conjecture} does.
\item Conjecture~\ref{la_conjecture} holds for all groups $G(m,2)$ for
  $\ell= 1$ (see Section~\ref{section.n2}), as well as for all
  $\ell$ when $m\leq 5$. For example, with $W=G(3,2)$, we get
 \begin{eqnarray*}
    \DW_q(\mathbf{t})&=&1+2\,h_1(\mathbf{t})+2\,h_2(\mathbf{t})+h_1^2(\mathbf{t})+h_3(\mathbf{t})+2\,h_2(\mathbf{t})h_1(\mathbf{t})\\
         &&\qquad +2\,h_4(\mathbf{t})+h_2^2(\mathbf{t})+3\,h_5(\mathbf{t})+2\,h_6(\mathbf{t})+h_7(\mathbf{t}).
  \end{eqnarray*}
\item Conjecture~\ref{la_conjecture} holds for all $\ell$, in the
  case $W=\sg=G(1,n)$ for $n\leq 4$. For example, we get the Hilbert
  series
  $$\mathcal{D}^{(\ell)}_{3,q}(\mathbf{t})=1+ 2\,h_1(\mathbf{t})+h_{11}(\mathbf{t})+h_2(\mathbf{t})+h_3(\mathbf{t}).$$
\item There seems to be an analogue of Conjecture~\ref{la_conjecture}
  for the subgroups $G(m,p,n)$ of $G(m,n)$ (see
  Section~\ref{section.Gmpn}); in particular,
  Conjecture~\ref{la_conjecture} holds for $n=2$ (including the
  dihedral groups $G(m,m,2)$) when $\ell=1$ (see
  Section~\ref{section.n2}), and for any $\ell$ for small values of
  $m,p,n$.
\end{enumerate}

Another interesting feature of the space $\DW_q$, is that it may be
characterized as the intersection of the kernels of a much smaller
family of operators than the set
\begin{equation}\label{famille_oper}
  \{ D_{q,\md}\ |\ \, |\md|=m\,k,\quad 1\leq k\leq n\}.
\end{equation}
Indeed, a straightforward calculation shows that the usual Lie-bracket
relation between the generators of the rational Steenrod algebra
generalize naturally:
\begin{equation}\label{formule_crochet}
  [D_{q,\md},D_{q,\mathbf{d'}}]=q(|\md|-|\mathbf{d'}|)\, D_{q,\md+\mathbf{d'}}.
\end{equation}
An efficient way to setup this calculation is to let both sides act
on the generating function for all monomials, namely the formal series
     $$\exp(Z.X)=\sum_{\md\in \NN^{\ell\times n}} \mx^\md\frac{\mx[z]^\md}{\md!},$$
 where $Z$ stands for a matrix of variables just as $X$ does, and $Z.X:=\sum_{ij} z_{ij}x_{ij}$.
It follows from~(\ref{formule_crochet}) that a polynomial is in the kernel of $D_{q,\md+\mathbf{d'}}$, whenever it lies in the kernels of both
$D_{q,\md}$ and $D_{q,\mathbf{d'}}$. From this, we can
immediately deduce that
\begin{proposition}\label{prop_crochet}
  The space of $q$-harmonic polynomials for $W=G(m,n)$ can be obtained as
  $$ \DW_q=\bigcap_{|\md|=m\ {\rm or}\ 2\,m} \mathrm{Ker}(D_{q,\md}).$$
\end{proposition}
For example, when $\ell=1$, and as already noted
in~\cite{Hivert_Thiery.SA.2002} in the case $W=\sg$, the space
$\qHW{W}$ is defined by just two linear differential equations:
\begin{equation*}
  \qHW{W}={\mathrm Ker}(D_{q,m})\cap {\mathrm Ker}(D_{q,2\,m})\,.
\end{equation*}
Similarly, when $\ell=2$, the space $\DW[2]_q$ is the intersection of the kernels
of only five operators:
$$D_{q,(1,0)},\qquad D_{q,(2,0)},\qquad D_{q,(1,1)},\qquad D_{q,(0,1)},\quad{\rm and}\quad D_{q,(0,2)}.$$
This is striking because, assuming that Conjecture~\ref{la_conjecture}
holds, a direct calculation of this joint kernel and a specialization
at $q=0$ would yield back the famous space of diagonally harmonic
polynomials $\mathcal{D}^{(2)}$. Yet, even if the mysterious structure
of $\mathcal{D}^{(2)}$ has been extensively studied in the past 20
years (see \cite{Haiman.2002}), no nice Gr\"obner basis for the ideal
$I_W^{(2)}$ is known, even for $W=\sg$.

It is also noteworthy that systematic variations on the main conjecture in~\cite{Hivert_Thiery.SA.2002} have been extensively studied in~\cite{Bergeron_Garsia_Wallach.2010}. In an upcoming work, we plan to describe how these variations may be adapted to the context of the reflection groups considered here, including the diagonal point of view.  In particular, since there is a close tie ({\it loc. sit.}) between the case $\ell=1$, with $W=\sg[n]$, and Wood's conjecture (stated in \cite{wood.psa.1998} or \cite{wood.hpsa.2001}), we also plan to analyze how to generalize it to our new expanded context.

\section{Deformed harmonic polynomials for $G(m,p,n)$}
\label{section.Gmpn}

This section presents work in progress toward generalizing the
construction of $q$-harmonic polynomials, and
Conjecture~\ref{la_conjecture}, to all finite complex reflection
groups.
For simplicity, we restrict ourselves to a single set of variables:
namely $\ell=1$. However, computer calculations suggest that the
extension to the diagonal case is straightforward.

Recall that all but a small number of finite complex reflection groups are part of an
infinite family of natural subgroups of the generalized symmetric
groups which we consider now. For $m,n\in \NN$, let $p$ be a divisor of
$m$. Then, the complex reflection group $G(m,p,n)$ is defined as:
\begin{displaymath}
  G(m,p,n) := \{ g \in G(m,n) \suchthat \det g^{m/p} = 1 \}\,.
\end{displaymath}
In particular, setting $p=1$, we get back $G(m,n)=G(m,1,n)$.  Recall, for example,  that the
classical dihedral groups correspond to the family $G(m,m,2)$.

The invariant ring for $W=G(m,p,n)$ is obtained by adjoining
$e_n^{m/p}$ to the invariant ring of $G(m,n)$, with
$e_n=e_n(\mx):=x_1\cdots x_n$ standing for the product of the
variables. It is well known that the invariant ring $\kk[\mx]^W$ may
then be described as the graded free commutative algebra:
\begin{displaymath}
   \kk[\mx]^W=\kk[p_m,\dots,p_{(m-1)n},e_n^{m/p}]\,,
\end{displaymath}
with $\deg(p_k)=k$ and $\deg(e_n^{m/p})=nm/p$. Notice the necessary suppression of the generator $p_{mn}$
for this presentation to be free.

The choice of a canonical $q$-analogue of $e_n^{m/p}$ does not appear to be
straightforward. Indeed, and as far as we know (see the discussion
in~\cite[Section 7.1]{Hivert_Thiery.SA.2002}), there is no natural
analogue of the elementary symmetric polynomial inside the rational
Steenrod algebra $\qsteen$.
Besides, experienced gained
in~\cite{Bergeron_Garsia_Wallach.2010} suggests that (generically) any choice of
$q$-analogue would give an isomorphic space. We therefore take the simplest option, which
is to not deform $e_n^{m/p}$ at all. Hence, for $W=G(m,p,n)$, we define the \emph{$q$-deformed
  rational $W$-Steenrod algebra} as
\begin{displaymath}
  \qsteen^{W} := \kk[P_{q,m},\dots,P_{q,mn},e_n^{m/p}].
\end{displaymath}
Accordingly, we obtain the graded space $\qHW{W}$ of the
$q$-harmonic polynomials for $W=G(m,p,n)$, just as
previously. Specifically, writing $\e$ for the operator $e_n(\PX)$,
$\qHW{W}$ is the intersection of $\qHW{G(m,n)}$ and $\mathrm{Ker}
(\e^{m/p})$. A natural question here is to ask whether
Conjecture~\ref{la_conjecture} holds for $G(m,p,n)$.  We will show in
Section~\ref{section.n2} that it does for $n=2$ and $\ell=1$.

A first step to confirm the choice of $e_n^{m/p}$ would be to prove
the following conjecture.
\begin{conjecture}
  \label{conjecture.e}
  The $q$-harmonic polynomials for $G(m,1,n)$, as defined above,
  coincide with those for $G(m,n)$.
\end{conjecture}
An equivalent but more concrete condition is that no $q$-harmonic
polynomial for $G(m,n)$ shall contain a monomial divisible by $e_n^m$.
This property holds for $n=2$, and all $q$-harmonic polynomials for
$\sg$ calculated in~\cite{Hivert_Thiery.SA.2002}.

In fact, we expect $\qHW{G(m,n)}$ to decompose as a direct sum of
$m$ layers $L_0(q)\oplus\cdots\oplus L_{m-1}(q)$ such that all the
elements of $L_k(q)$ are divisible by $e_n^k$ but not by $e_n^{k+1}$,
as in Figure~\ref{layers_graph}. Furthermore, $\e$ (depicted by the
gray down arrows in this figure) would be an isomorphism from $L_k(q)$
to $L_{k-1}(q/(1+q))$, the change of $q$ being due to the equation
\begin{equation}
  e_n(\PX) D_{q,k}  = (1+q) D_{\frac{q}{1+q},k} e_n(\PX)\,.
\end{equation}
In that case, the $q$-harmonic polynomials for $G(m,p,n)$ would be
given by
\begin{equation}
  \qHW{G(m,p,n)} = L_0(q)\oplus \cdots \oplus L_{m/p-1}(q)\,.
\end{equation}
For example, in Figure~\ref{layers_graph}, the $q$-harmonic
polynomials for the dihedral group $G(4,1,2)$ are given by
$\qHW{G(4,1,2)}=L_0(q)$. Similarly, $\qHW{G(4,2,2)}=L_0(q)\oplus
L_1(q)$. We expect that, in general, the space $\qHW{G(m,n)}$ consists of
$m$ copies of $\qHW{G(m,n,1)}$, that may be
constructed  by ``lifting back'' through $\e$.

This further suggests that the lack of operators commuting
appropriately with the action of the rational Steenrod algebra, as
reported in~\cite[Section 7]{Hivert_Thiery.SA.2002} can be
circumvented by allowing $q$ to change during the commutation. For
example, for $m=\ell=1$, the harmonic polynomials are usually
constructed from the Vandermonde determinant by iterative applications
of operators $\p_i$; it would be worth finding analogues of
those operators which would construct new $q$-harmonic polynomials
from $q'$-harmonic polynomials for some other $q'$.

\section{Inflating $q$-harmonic polynomials from $\sg$ to $G(m,n)$ and $G(m,m,n)$}

\TODO{Fran\c{c}ois: proofread}

In this section, we do some preliminary steps in the following direction.
\begin{problem}
  Assume that a basis of the $q$-harmonic polynomials for $\sg[n]$ is
  given. Is it possible to construct from it a basis of the
  $q$-harmonic polynomials for $G(m,n)$? for $G(m,m,n)$?
\end{problem}
Beware that the analogous problem for constructing diagonally
$q$-harmonic polynomials from $q$-harmonic polynomials is already hard
at $q=0$.


To start with, any $q$-harmonic polynomial for $W=\sg$ remains
$q$-harmonic for $G(m,n)$. It also remains $q$-harmonic for $G(m,p,n)$
as soon as Conjecture~\ref{conjecture.e} holds for $\sg$.
We now construct more $q$-harmonic polynomials by inflating those of
$\sg$. To this end, we consider the inflation algebra morphism and its
analogue (which is just a linear morphism) on the dual basis:
\begin{equation}
  \inflate_r:
  \begin{cases}
    \kk[X] & \hooktwoheadrightarrow \kk[X^r] \\
    \mx^{\md} & \mapsto \mx^{r\md}
  \end{cases}
  \qquad
  \qquad
  \inflatep_r:
  \begin{cases}
    \kk[X] & \hooktwoheadrightarrow \kk[X^r] \\
    \mx^{(\md)} & \mapsto \mx^{(r\md)}
  \end{cases}\,,
\end{equation}
where, by a slight notational abuse, $X^r$ stands for the matrix of
the $r$-th powers of the variables and
$\mx^{(\md)}:=\frac{1}{\md!}\mx^\md$.\TODO{define this notation in the introduction?}

\begin{proposition}
  \label{proposition.inflate}
  Let $W=G(m,n)$ and $r$ be a divisor of $m$. Then, the morphism
  $\inflatep_r$ restricts to a graded $\sg$-module embedding
  (resp. isomorphism if $r=m$) from $\qHW{\sg}$ to
  $\qHW[q/m]{W}\cap\kk[X^r]$, up to an appropriate $r$-scaling of
  the grading.
  The statement extends to any $G(m,p,n)$ as soon as
  Conjecture~\ref{conjecture.e} holds for $\sg$.
\end{proposition}

The first step toward this proposition is to define an appropriate
inflation on the $q$-rational Steenrod algebra. Note that the
operators $P_{q,k}$ live inside the subalgebra $\kk[\XPX, \X]$
of the Weyl algebra, where $\XPX$ denotes the matrix of the
Euler operators $x\p_x$ for $x\in X$. The only non-trivial brackets in
this algebra are $[x\p_x, x] = x$, for $x\in X$, from which it follows
that $[x\p_x, x^k] = k x^k$. Similarly, the operators $D_{q,k}$ live
inside the subalgebra $\kk[XPX, \PX]$, with analogue relations.
\begin{remark}
  \label{phi_scale}
  Fix $r\in \NN$. Then, the two mappings
  \begin{displaymath}
    x\p_x \mapsto 1/r(x\p_x), \quad x\mapsto x^r, \text{ for } x\in X
    \qquad \text{ and } \qquad
    x\p_x \mapsto 1/r(x\p_x), \quad \p_x\mapsto \p_x^r, \text{ for } x\in X
  \end{displaymath}
  extend respectively to algebra isomorphisms
  \begin{displaymath}
    \Inflate_r:  \kk[\XPX,\X ]\hooktwoheadrightarrow \kk[\XPX, \X ^m]
    \qquad \text{ and } \qquad
    \Inflatep_r: \kk[\XPX,\PX]\hooktwoheadrightarrow \kk[\XPX, \PX^m]\,.
  \end{displaymath}
  Furthermore, those isomorphisms are compatible with the
  action on inflated polynomials: for $f\in \kk[X]$ and $F$ in
  $\kk[\XPX, \X]$ (resp. in $\kk[\XPX,\PX]$), we have
  \begin{equation}
    \label{equation.inflate}
    \Inflate_r(F) ( \inflate_r(f) ) = \inflate_r(F(f))
    \qquad \text{ and } \qquad
    \Inflatep_r(F) ( \inflatep_r(f) ) = \inflatep_r(F(f)) \,.
  \end{equation}
\end{remark}
Using this remark, a straightforward calculation shows that:
\begin{equation}
  \Inflate_r(P_{q,\md}) = P_{q/r,r\md} \qquad \text{ and } \qquad \Inflatep_r(D_{q,\md}) = D_{q/r,r\md}\,.
\end{equation}
This readily implies that $\Inflate_m$ restricts to an isomorphism
from the $q$-rational Steenrod algebra for $\sg$ to that for
$G(m,n)$. Computer exploration suggests that the Gr\"obner basis for
the right ideal generated by the Steenrod algebra for $G(m,n)$ is
simply obtained by inflating that for $\sg$. This possibly opens the
door for controlling the leading monomials of ``$q$-hit polynomials''
for $G(m,n)$ from those for $\sg$.




Returning to our main goal, we now have all the ingredients to prove
Proposition~\ref{proposition.inflate}.\\
\begin{proof}[of Proposition~\ref{proposition.inflate}]
  Let $f\in \kk[\X]$. Then, using Equation~\eqref{equation.inflate},
  \begin{displaymath}
    D_{q/r,\, r\md}(\inflate(f)) = \Inflatep_r(D_{q,\md}) (\inflate(f)) =
    \inflate( D_{q,\md}(f) )\,.
  \end{displaymath}
  Hence $D_{q/r, r\md}(\inflate(f))=0$ if and only if
  $D_{q,\md}(f)=0$. The statements follows.
\end{proof}

\section{Singular values}

As discussed in~\cite{Hivert_Thiery.SA.2002}, the statement of
Conjecture~\ref{la_conjecture} may fail when $q$ is not a formal
parameter; in that case, $q$ is called \emph{singular}. Computer
exploration (see Table~A.3 of~\cite{Hivert_Thiery.SA.2002}) and the
complete analysis of the case $n=2$ suggested that the only such
singular values for $W=\sg$ and $\ell=1$ are negative rational numbers
of the form $-a/b$ with $a\leq n$. In~\cite{DAdderio_Moci.2010}
D'Adderio and Moci refined this statement to $a\leq n \leq b$ (with
$q=-a/b$ not necessarily reduced), and proved that all such values are
indeed singular by constructing explicit exceptional $q$-harmonic
polynomials. 

\begin{proposition}
  \label{proposition.singular}
  Let $W=G(m,n)$, $\ell$ be arbitrary, and $q=-a/b$ with $a\leq n \leq
  b$. Then, $q$ is singular.
\end{proposition}
\begin{proof}[(sketch of)]
  Let $f(x_1,\dots,x_n)$ be the $q$-harmonic polynomial for $\sg[n]$
  constructed in~\cite{DAdderio_Moci.2010}. Then, as stated in
  Proposition~\ref{proposition.inflate}, $f(x_1^m,\dots,x_n^m)$ is a
  $q/m$-harmonic for $W$ of high enough degree (as
  in~\cite{DAdderio_Moci.2010}) to disprove the statement of
  Conjecture~\ref{la_conjecture}.
  Going from $\ell=1$ to $\ell$ arbitrary is then straightforward,
  since the intersection of $\DW[\ell]_q$ with the polynomial ring in
  the first set of variables is $\HW_q$.
\end{proof}

It is worth noting that, for $n=2$, and $\ell=1$ the singular values
are exactly those listed in Proposition~\ref{proposition.singular}
(see Section~\ref{section.n2}). However, at this stage, we lack
computer data to extend this to a conjecture for all $n$ and $\ell$.

\TODO{Can we argue about $G(m,p,n)$?}

\section{Complete study for $n=2$}
\label{section.n2}

\let\a=\alpha
\let\b=\beta


In this section, we prove Conjecture~\ref{la_conjecture} for any group
$W=G(m,p,2)$ in the case $\ell=1$.  We denote for short the two
variables $x$ and $y$ instead of $x_1$ and $x_2$. Naturally $\p_x$ and
$\p_y$ are the corresponding differential operators. We also introduce
the following $q$-analogue of the Pockhammer symbol $(d)_k$:
\begin{displaymath}
\tp{d}{k} :=d\,(d-1)\,\cdots\,(d-k+1)\, (1+q\,(d-k)).
\end{displaymath}
Then, for any monomial $\xy{\a}{\b}$, one has:
\begin{equation}
  \label{action_Dkq}
  D_{q,k} (\xy{\a}{b}) = \tp{\a}{k} \, \xy{\a-k}{\b} + \tp{\b}{k} \, \xy{\a}{\b-k}\,,
\end{equation}
which is well defined for any nonnegative numbers $\a$ and $\b$, since
$\tp{\a}{k}=0$ whenever $\a<k$.
\begin{remark}
  Let $W=G(m,m,2)$ be the dihedral group. Then, the space $\qHW{W}$
  coincides with $\HW$ whenever $q$ is not of the form $-1/b$ with
  $1\leq b\leq m$. A basis is given by
  \begin{equation}\label{base_diedral}
    \{ 1, x, x^2, x^3, \dots,x^{m-1}, x^m-y^m, y^{m-1}, y^{m-2}, \dots, y^2, y \}\,.
  \end{equation}
  Otherwise, the basis of $\qHW{W}$ contains additionally the
  monomials $x^{b+m}$ and $y^{b+m}$, or just the binomial $x^{2m} -
  y^{2m}$ if $b=m$.
\end{remark}
\begin{proof}
  Let $f=f(x,y)$ be an homogeneous $q$-harmonic polynomial in
  $\kk[x,y]$. It satisfies:
  \begin{displaymath}
    D_{q,m}(f) = 0, \qquad D_{q,2m}(f) =0, \qquad and \qquad \e(f)=0,
  \end{displaymath}
  where $\e=\partial_x\partial_y$. By the last equation, $f$ is of the
  form $\lambda x^d + \mu y^d$, and the two other equations rewrite
  as $(d)_k (\lambda x^{d-km} + \mu y^{d-km})$ for $k=1,2$.  The
  statement follows.
\end{proof}

\begin{proposition}
  For $W=G(m,2)$ and $q \neq -a/b, 1 \leqslant a \leqslant 2 \leqslant b$, the space
  $\qHW{W}$ is isomorphic as a graded $W$-module to
  $\HW$. In particular it is of dimension $2m^2$. A basis
  is given by
  \begin{equation}\label{la_base}
    \{\xy{\a}{\b}\}_{0 \leqslant a < m \atop 0 \leqslant b < m}
    \cup
    \{ \tp{\b+m}{m}\xy{\a+m}{\b} - \tp{\a+m}{m}\xy{\a}{\b+m} \}_{0 \leqslant \a < m \atop 0 \leqslant \b < m}\,.
  \end{equation}
\end{proposition}
\begin{proof}
  As suggested by Equation~\eqref{action_Dkq}, the implicit
  combinatorial ingredient is the length of the longest string
  $\dots,\xy{\a-m}{\b+m}, \xy{\a}{\b}, \xy{\a+m}{\b-m},\dots$ containing any
  given monomial.

  Obviously, whenever $\a<m$ and $\b<m$, the monomial $\xy{\a}{\b}$ is killed
  by both operators $D_{q,m}$ and $D_{q,2\,m}$, and is therefore
  $q$-harmonic. This gives the first $m^2$ monomials
  in~\eqref{la_base}.  Using Equation~\ref{action_Dkq}, a direct
  calculation shows that the remaining $m^2$ binomials
  \begin{displaymath}\label{formule}
    \tp{\b+m}{m} \xy{\a+m}{\b} - \tp{\a+m}{m} \xy{\a}{\b+m}
  \end{displaymath}
  for $\a<m$ and $\b<m$ are also $q$-harmonic.

  We now consider a monomial $\xy{\a'}{\b'}$ that does not appear in any
  of the $q$-harmonic polynomials of~\eqref{la_base}, and prove that
  it cannot appear in any other $q$-harmonic polynomial.  It is
  straightforward that we can choose $\a$ and $\b$ such that
  $$\xy{\a'}{\b'} \in \{ \xy{\a+m}{\b-m}, \xy{\a}{\b}, \xy{\a-m}{\b+m} \}\,.$$
  Let $h = c_1 \xy{\a+m}{\b-m} + c_2 \xy{\a}{\b} + c_3 \xy{\a-m}{\b+m} +
  \cdots $ be a $q$-harmonic polynomial. Then,
  \begin{displaymath}
    0 = D_{q,m}(h)_{|\xy{\a}{\b-m}} = c_1 \tp{\a+m}{m} + c_2 \tp{\b}{m}\,.
  \end{displaymath}
  Looking similarly at $D_{q,m}(h)_{|\xy{\a-m}{\b}}$ and $D_{q,
    2\,m}(h)_{|\xy{\a-m}{\b-m}}$, shows that the coefficients $c_1$,
  $c_2$, and $c_3$ must satisfy the following system of equations:
  \begin{displaymath}
    \begin{array}{rcrcrcc}
      \tp{\a+m}{m\phantom{2}}\,c_1 & + & \tp{\b}{m}\,c_2 & & &=& 0 \\
      & & \tp{\a}{m}\,c_2 & + & \tp{\b+m}{m\phantom{2}}\,c_3 &=& 0 \\
      \tp{\a+m}{2m}\,c_1 & & & + & \tp{\b+m}{2m}\,c_3 &=& 0 \\
    \end{array}
  \end{displaymath}
  whose determinant is:
  \begin{displaymath}
     \frac{(\a+m)!(\b+m)!}{(\a-m)!(\b-m)!}(1+q\,(\a-m))(1+q\,(\b-m))(2+q\,(\a+\b))\,.
  \end{displaymath}
  Therefore $c_1=c_2=c_3=0$ whenever $q$ is not one of the announced singular values.
\end{proof}

  \begin{figure}[!h]
    \label{layers_graph}
    \centering
    \begin{tikzpicture}[xscale=1.6]
      \tikzstyle{every node}=[fill=white,minimum size=5pt,inner sep=1pt]
      \tikzset{
        vertical/.style={->, gray, shorten >=1pt},
        rededge/.style={->, red, shorten >=1pt, very thick},
        blueedge/.style={->, blue, shorten >=1pt, very thick},
        monomials/.style={-, line width=0pt, darkgreen, shorten >=1pt}
              }

      \node at (-8, 4) {$L_0(q):$};
      \node at (-8, 6) {$L_1(q):$};
      \node at (-8, 8) {$L_2(q):$};
      \node at (-8,10) {$L_3(q):$};

      \draw (0,0) node (00) {00}
         ++ (-2,1) node (10-*) {10-$\ast$}
         ++ (-2,1) node (20-*) {20-$\ast$}
         ++ (-2,1) node (30-*) {30-$\ast$}
         ++ (-1,1) node (40-04) {40-04}
         ++ (2,-1) node (*-03) {$\ast$-03}
         ++ (2,-1) node (*-02) {$\ast$-02}
         ++ (2,-1) node (*-01) {$\ast$-01};
      \draw[blueedge] (10-*) -- (00);
      \draw[blueedge] (20-*) -- (10-*);
      \draw[blueedge] (30-*) -- (20-*);
      \draw[blueedge] (40-04) -- (30-*);
      \draw[rededge] (40-04) -- (*-03);
      \draw[rededge] (*-03) -- (*-02);
      \draw[rededge] (*-02) -- (*-01);
      \draw[rededge] (*-01) -- (00);

      \draw (0,2) node (11) {11}
         ++ (-2,1) node (21-*) {21-$\ast$}
         ++ (-2,1) node (31-*) {31-$\ast$}
         ++ (-2,1) node (41-05) {41-05}
         ++ (-1,1) node (51-15) {51-15}
         ++ (2,-1) node (50-14) {50-14}
         ++ (2,-1) node (*-13) {$\ast$-13}
         ++ (2,-1) node (*-12) {$\ast$-12};
      \draw[blueedge] (21-*) -- (11);
      \draw[blueedge] (31-*) -- (21-*);
      \draw[blueedge] (41-05) -- (31-*);
      \draw[blueedge] (51-15) -- (41-05);
      \draw[rededge] (51-15) -- (50-14);
      \draw[rededge] (50-14) -- (*-13);
      \draw[rededge] (*-13) -- (*-12);
      \draw[rededge] (*-12) -- (11);

      \draw (0,4) node (22) {22}
         ++ (-2,1) node (32-*) {32-$\ast$}
         ++ (-2,1) node (42-06) {42-06}
         ++ (-2,1) node (52-16) {52-16}
         ++ (-1,1) node (62-26) {62-26}
         ++ (2,-1) node (61-25) {61-25}
         ++ (2,-1) node (60-24) {60-24}
         ++ (2,-1) node (*-23) {$\ast$-23};
      \draw[blueedge] (32-*) -- (22);
      \draw[blueedge] (42-06) -- (32-*);
      \draw[blueedge] (52-16) -- (42-06);
      \draw[blueedge] (62-26) -- (52-16);
      \draw[rededge] (62-26) -- (61-25);
      \draw[rededge] (61-25) -- (60-24);
      \draw[rededge] (60-24) -- (*-23);
      \draw[rededge] (*-23) -- (22);

      \draw (0,6) node (33) {33}
         ++ (-2,1) node (43-07) {43-07}
         ++ (-2,1) node (53-17) {53-17}
         ++ (-2,1) node (63-27) {63-27}
         ++ (-1,1) node (73-37) {73-37}
         ++ (2,-1) node (72-36) {72-36}
         ++ (2,-1) node (71-35) {71-35}
         ++ (2,-1) node (70-34) {70-34};
      \draw[blueedge] (43-07) -- (33);
      \draw[blueedge] (53-17) -- (43-07);
      \draw[blueedge] (63-27) -- (53-17);
      \draw[blueedge] (73-37) -- (63-27);
      \draw[rededge] (73-37) -- (72-36);
      \draw[rededge] (72-36) -- (71-35);
      \draw[rededge] (71-35) -- (70-34);
      \draw[rededge] (70-34) -- (33);

      \draw[vertical] (11) -- (00);
      \draw[vertical] (21-*) -- (10-*);
      \draw[vertical] (31-*) -- (20-*);
      \draw[vertical] (41-05) -- (30-*);
      \draw[vertical] (51-15) -- node[left] {$\e\ $} (40-04);
      \draw[vertical] (50-14) -- (*-03);
      \draw[vertical] (*-13) -- (*-02);
      \draw[vertical] (*-12) -- (*-01);

      \draw[vertical] (22) -- (11);
      \draw[vertical] (32-*) -- (21-*);
      \draw[vertical] (42-06) -- (31-*);
      \draw[vertical] (52-16) -- (41-05);
      \draw[vertical] (62-26) -- node[left] {$\e\ $} (51-15);
      \draw[vertical] (61-25) -- (50-14);
      \draw[vertical] (60-24) -- (*-13);
      \draw[vertical] (*-23) -- (*-12);

      \draw[vertical] (33) -- (22);
      \draw[vertical] (43-07) -- (32-*);
      \draw[vertical] (53-17) -- (42-06);
      \draw[vertical] (63-27) -- (52-16);
      \draw[vertical] (73-37) -- node[left] {$\e\ $} (62-26);
      \draw[vertical] (72-36) -- (61-25);
      \draw[vertical] (71-35) -- (60-24);
      \draw[vertical] (70-34) -- (*-23);

      \draw (-7,3) node (BotLeft) {};
      \draw (0,7) node (TopRight) {};
      \draw[monomials] (BotLeft) -- (TopRight);

    \end{tikzpicture}
    \caption{Structure of $q$-Harmonic polynomials of $G(4,2)$.  For
      short, the $q$-harmonic binomial $\tp{d}{4} \xy{\a}{\b} -
      \tp{a}{4} \xy{\a'}{\b'}$ is denoted ``$\a\b$-$\a'\b'$''. Similarly, the
      $q$-harmonic monomial $\xy{\a}{\b}$ is denoted ``$\a\b$'',
      ``$\a\b$-$\ast$'', or ``$\ast$-$\a\b$''. The blue (resp. red) arrows
      denote the action of the would be $q$-analogues of the operators
      $\p_x$ and $\p_y$ within each layer $L_i$. The gray
      arrows denote the action of the operator $\e = e_2(\PX)
      = \p_x\p_y$ (recall that it changes the value of
      $q$).  The green line separates the $q$-harmonic monomials and
      binomials.}

  \end{figure}

\begin{corollary}
  For $W=G(m,p,2)$ and $q \neq -a/b, 1 \leqslant a \leqslant 2
  \leqslant b$, the space $\qHW{W}$ is isomorphic as a graded
  $W$-module to $\HW$. Its basis is obtained by considering the layers
  $L_0(q),\ldots,L_{m/p-1}(q)$ of the $q$-harmonic polynomials for
  $G(m,2)$.
\end{corollary}
\begin{proof}
  Select out of Equation~\eqref{la_base} the elements which satisfy
  the extra equation $\e^{m/p}(f) = 0\,$. See also, in
  Figure~\ref{layers_graph}, the vertical arrows which depict the
  action of $\e$.
\end{proof}



\TODO{un seul Wood en ref?}

\bibliographystyle{alpha}

\bibliography{main}

\end{document}